\documentclass{article}

\usepackage{amsmath,amssymb,amsthm}
\usepackage{enumitem}
\setlist[enumerate]{label=(\roman*), leftmargin=*}
\usepackage{dsfont}
\usepackage{xcolor}
\usepackage{natbib}
\usepackage{hyperref}

\numberwithin{equation}{section}
\setlist[enumerate]{leftmargin=*,label=(\roman*)}

\theoremstyle{plain}
\newtheorem{theorem}{Theorem}[section]
\newtheorem{proposition}[theorem]{Proposition}
\newtheorem{lemma}[theorem]{Lemma}
\newtheorem{corollary}[theorem]{Corollary}

\theoremstyle{remark}
{\theoremstyle{definition}\newtheorem{definition}[theorem]{Definition}}
\newtheorem{example}[theorem]{Example}

\newtheorem{remark}[theorem]{Remark}

\renewcommand{\epsilon}{\varepsilon}
\newcommand{\eps}{\varepsilon}
\renewcommand{\phi}{\varphi}

\newcommand{\cC}{\mathcal{C}}

\newcommand{\cF}{\mathcal{F}}

\newcommand{\cI}{\mathcal{I}}

\newcommand{\cL}{\mathcal{L}}
\newcommand{\cM}{\mathcal{M}}

\newcommand{\cP}{\mathcal{P}}

\newcommand{\N}{\mathbb{N}}

\newcommand{\R}{\mathbb{R}}

\renewcommand{\d}{{\rm d}}

\newcommand{\Rd}{{\R^d}}
\newcommand{\one}{\mathds{1}}

\DeclareMathOperator{\Divergence}{D}

\DeclareMathOperator{\ex}{ex}

\newcommand{\Div}[3]{\Divergence_{#1}{\!\left(\left. #2 \,\right\rVert #3 \right)}}

\title{Extreme points of sets of probability measures and $\phi$-divergences}
\author{Gerrit Bauch\thanks{Center for Mathematical Economics, Bielefeld University, PO Box 10 01 31, 33501 Bielefeld, Germany. Email: \href{mailto:gerrit.bauch@uni-bielefeld.de}{gerrit.bauch@uni-bielefeld.de}}, Max Nendel\thanks{Department of Statistics and Actuarial Science, University of Waterloo, 200 University Ave W, Waterloo, ON N2L 3G1, Canada. Email: \href{mailto:mnendel@uwaterloo.ca}{mnendel@uwaterloo.ca}}, Alessandro Sgarabottolo\thanks{Department of Mathematics, LMU M\"unchen, Theresienstr.\ 39, 80333 Munich, Germany. Email: \href{mailto:sgarabottolo@math.lmu.de}{sgarabottolo@math.lmu.de}}}
\date{\today}

\begin{document}

\maketitle

\begin{abstract}
    In this work, we prove several equivalent characterizations of the extreme points of convex sets of probability measures of the form $\cM=\cP\cap H$, where $\cP$ denotes the set of all probability measures on an arbitrary measurable space $(\Omega,\cF)$ and $H$ is an affine set of signed measures on $(\Omega,\cF)$ with finite variation.\ 
    We first give a precise measure-theoretic formulation of the heuristic that extreme measures have minimal support.
    We then connect this with the notion of minimality with respect to absolute continuity, and prove that points that dominate no other element of $\cM$ are the only ones realizing the blow-up of a certain divergence map for any suitable $\phi$-divergence.\ Finally, considering a different class of $\phi$-divergences, we recover a characterization of the extreme points of $\cM$ as strict local maximizers of $\phi$-divergences relative to any suitable dominating measures.\
    We apply our result to recover and complement results from the literature in the context of finite spaces, sets of measures defined by integral constraints, multi-marginal couplings, and dominated sets of probability measures.\smallskip
\end{abstract}

\section{Introduction}

Extreme points encode the irreducible geometry of a convex set.\ A well-known fact is that an element is extreme if and only if it cannot be moved in two opposite feasible directions and remain in the set.\ For sets of probability measures, this geometric condition has traditionally been studied in relation to the support of the measure.\ In finite dimensions, the vertices of polyhedra of the form $\{x\in\R^n\mid x\geq 0,\ Ax=b\}$ are described by their active coordinates and a linear-independence condition, see, for example, \citet[Chapter~8]{schrijver1986theory}.\ Analogous principles based on minimality of the support underlie sets of probability measures satisfying finitely many integral constraints.\ These were studied by \citet{karr1983extreme} and \citet{winkler1988extreme} building on the functional-analytic characterizations of \citet{douglas1964extremal,douglas1966extremal-II}.\ Moreover, in the case of a finite state space, general discrete marginal problems can be analyzed through the combinatorics of their supports, see \citet{denny1980support}.\ Related to marginal constraints and optimal transport on general measurable spaces, graph structures, acyclicity, and uniqueness of a coupling on its support have played a central role since \citet{letac1966representation}, see also \citet{benes1987support,hestir1995supports,moameni2016supports,dallaglio1991advances,Villani2009OptimalTransport}.

These results motivate the informal principle that an extreme measure should have ``minimal support''.\ The difficulty is to formulate this principle in a way that is both intrinsic and valid beyond finite or purely atomic spaces.\ 

A strong formulation might say that an extreme point $p$ admits no distinct feasible measure $p'$ with $p'\ll p$. In the setting of doubly stochastic measures, this conjecture was attributed to Feldman, cf.\ \citet[p.\ 194]{dallaglio1991advances}, and has been proven to be false.\ For example, \citet{losert1982counterexamples} constructs an extreme doubly stochastic measure that dominates a distinct doubly stochastic measure.\ Thus extremality and minimality under absolute continuity need not coincide, even with the additional structure provided by marginal constraints.

Needless to say, characterizing the extreme points of a set of probability measures is naturally appealing. For example, whenever a convex set is compact in a suitable topology, Bauer's maximum principle, cf.\ \citet[Theorem 7.69]{AliprantisBorder2006infinite}, ensures that a continuous linear functional attains both its minimum and its maximum at an extreme point.
More generally, whenever the integral representation property holds, every measure can be written via a barycentric formula over the set of extremal measures, cf.\ \cite{Weizsacker1979IntegralRepresentation} for a detailed discussion.\
Many optimization problems can therefore be reduced to the, typically much smaller, set of extreme points.
More recently, the fast development of Distributionally Robust Optimization has brought renewed interest in this topic, in particular in the context of nonparametric sets of measures defined by optimal transport distances, see, for instance, \cite{GaoKleywegt2023distributionally,EsfahaniKuhn2018data,Wozabal2012framework,ZhaoGuan2018datadriven}. For example, \cite{Wozabal2012framework} shows that the extreme points of a Wasserstein ball around a measure concentrated on $n$ atoms are concentrated on at most $n+3$ atoms. This result was refined by \cite{OwhadiScovel2017extremepoints} under additional conditions, and the atomic structure of the extreme points of a Wasserstein ball around an empirical measure was extensively used by \cite{EsfahaniKuhn2018data}.

In this paper, we study extremality for general sets of the form $\cM=\cP\cap H$, where $\cP=\cP(\Omega)$ denotes the set of all probability measures on a  measurable space $(\Omega,\cF)$ and  $H$ is an affine set of signed measures on $(\Omega,\cF)$ with finite variation.\ Our main result, Theorem~\ref{thm:extreme:divergence}, provides several equivalent descriptions for extreme points of the set $\cM$.\ The second item in Theorem~\ref{thm:extreme:divergence} identifies the weakening of minimality under absolute continuity that restores an equivalence with extremality.\
That is, we prove that $p\in\cM$ is extreme if and only if there is no $p'\in\cM\setminus\{p\}$ such that $p'\ll p$ and $\d p'/\d p$ is essentially bounded.\ The boundedness requirement is both natural and sharp, as we show in Example \ref{ex:extreme:not:minimal}.\ Intuitively, a bounded density allows one to move a positive distance from $p$ in the direction $p-p'$ while remaining non-negative, and every non-trivial convex decomposition of $p$ automatically produces a bounded density.\ This criterion shows why the finite and infinite cases differ, namely, on a finite space every density is bounded, so extremality is equivalent to minimality under absolute continuity.

Another main theoretical contribution establishes a connection between extremality and $\varphi$-divergences, introduced in the seminal papers by \citet{zbMATH03322635,zbMATH03252891,zbMATH03252892}.\ We start by observing that all $\ll$-minimal elements of a set $\cM = \cP \cap H$ are extremal and admit a characterization through the divergence associated with a convex function $\phi$ with $\phi(1)=0$ and $\phi'_+(0)=-\infty$, see Theorem~\ref{thm:minimal:divergence}.\ In particular, if $p$ is $\ll$-minimal, for every $p' \in \cM\setminus\{p\}$ and every common dominating measure $q \in \cM$, $\lambda = 1$ is a strict local maximizer of the map
\begin{equation}\label{eq:intro:divergence}
\lambda \mapsto \Div{\phi}{\lambda p +(1-\lambda)p'}{q},
\end{equation}
where, for $p,q \in \cP$, $\Div{\phi}{p}{q}$ is the $\phi$-divergence of $p$ relative to $q$.\
Although $\ll$-minimal elements of $\cM$ do not exhaust the sets of extreme points of $\cM$, the last item of Theorem \ref{thm:extreme:divergence} recovers such a characterization upon considering a different class of $\phi$-divergences and restricting to suitable admissible dominating measures $q \in \cM$.

We illustrate our results in several settings.\ On finite sets $\Omega$, where extremality and $\ll$-minimality agree, the characterization of extreme points by the strict local maximization of the map \eqref{eq:intro:divergence} holds for any divergence with infinite negative slope at zero and any common dominating reference measure, thus extending the divergence criterion of \citet{bauch2025correlationuncertaintydecisiontheoreticapproach} beyond the reference independent coupling and beyond marginal constraints.
For finite integral constraints, the bounded-perturbation criterion directly implies the familiar support bound of at most $n+1$ atoms and the associated linear-independence condition from \citet{karr1983extreme,winkler1988extreme}.
Moreover, we obtain the equivalence with $\ll$-minimality and, therefore, an alternative characterization via $\phi$-divergences. For couplings on general spaces, concentration on a marginal uniqueness set implies $\ll$-minimality. This covers Monge couplings, quadratic optimal transport plans obtained by Brenier's theorem \citep{brenier1991polar}, the strictly convex displacement costs studied by \citet{gangbo1996geometry}, and one-dimensional monotone couplings, see, e.g., \citet[Lecture 4.6]{ambrosio2021lectures}.\
Finally, we describe the extreme points of dominated families of measures and explain the corresponding implication for sets of one-period equivalent martingale measures.

The rest of the paper is organized as follows.\ Section~\ref{sec:extreme} proves the general characterizations in terms of bounded densities and divergences, and Section~\ref{sec:applications} applies them to finite spaces, integral constraints sets, multi-marginal couplings, and dominated families of probability measures.

\section{Setup and main results} \label{sec:extreme}
Throughout, let $(\Omega, \cF)$ be a measurable space, $\cP = \cP(\Omega)$ denote the set of probability measures on $\Omega$, and $\cM = \cP \cap H$ be the nonempty intersection of $\cP$ with an affine space $H$ of signed measures on $\Omega$ with finite total variation, cf.\ \citet[Chapter 6, pp.\ 116--119]{rudin1987real}, \citet[Chapter III.1, pp.\ 95--99]{dunfordschwartz1988}.
For a convex set of probability measures $\cC$, we write $\ex \cC$ for the set of its extreme points, i.e., those points in $\cC$ that cannot be written as a proper convex combination of two distinct points in $\cC$. More precisely, if $p \in \ex \cC$ and $p = \lambda p' + (1-\lambda)p''$ with $\lambda \in (0,1)$, then $p'=p''=p$.

As usual, we write $p\ll q$ for two probability measures $p,q\in \cP$ if $q(A)=0$ implies $p(A)=0$ for all $A\in \cF$.\ In this case, we write $\frac{\d p}{\d q}$
for the $q$-a.s.\ unique Radon-Nikod\'ym derivative or density $\Omega\to [0,\infty)$ of $p$ with respect to $q$.\ For two probability measures $p,q\in \cP$, we also write $p \approx q$ if $p$ and $q$ are equivalent, that is $p \ll q$ and $q \ll p$.

A recurring heuristic in the study of extreme points of sets of probability measures is that they should have a minimal support.\ Motivated by this idea, we define the notion of $\ll$-minimality.

\begin{definition}\label{def:minimal}
    An element $p\in\cM$ is called \emph{$\ll$-minimal in $\cM$} if there exists no $p'\in\cM\setminus\{p\}$ with $p'\ll p$.
\end{definition}

Despite the intuitive connection between $\ll$-minimality and extremality, we observe that an equivalence does not hold in general, as the following example shows.
\begin{example}\label{ex:extreme:not:minimal}
    Take $\Omega=[0,1]$ with the Borel $\sigma$-algebra, and denote by $\mu$ the Lebesgue measure on $[0,1]$.\ For $t\in\R$, let $p_t$ be the finite signed measure with Lebesgue density $1+t(2x-1)$, and let $H:=\{p_t\,\vert\,t\in\R\}$.\ Since $1+t(2x-1)\geq 0$ for $\mu$-a.e.\ $x\in[0,1]$ if and only if $t\in[-1,1]$, we have
    \[
    \cM=\cP\cap H=\big\{p_t\,\vert\,t\in[-1,1]\big\}.
    \]
    Let $p:=p_1$ or $p:=p_{-1}$ with density $2x$ or $2(1-x)$, respectively.\ Being an endpoint of the segment $\cM$, $p$ is an extreme point of $\cM$.\ However, $p$ is not $\ll$-minimal in $\cM$. As a matter of fact, since the density $1+t(2x-1)$ is $\mu$-a.s.\ strictly positive for every $t\in[-1,1]$, all elements of $\cM$ are equivalent to $\mu=p_0$, and in particular $p_t\ll p$ for all $t\in[-1,1]$.
\end{example}

The following proposition identifies the exact gap between the two notions of extremality and $\ll$-minimality.\ While, for $\cM=H\cap \cP$, $\ll$-minimality implies extremality, extremality only excludes dominated measures with essentially bounded density.

\begin{proposition}\label{prop:extreme:minimal}
   For an element $p \in \cM$, the following statements are equivalent.
    \begin{enumerate}
        \item $p$ is an extreme point of $\cM$.\label{it:prop:extreme}
        \item \label{it:prop:bounded} There exists no $p' \in \cM\setminus \{p\}$ with $p'\ll p$ and $p$-a.s.\ bounded density $\frac{\mathrm{d}p'}{\mathrm{d}p}$.
    \end{enumerate}
\end{proposition}

\begin{proof}
\ref{it:prop:bounded} $\Rightarrow$ \ref{it:prop:extreme}:\
    If $p$ is not an extreme point of $\cM$, there exist $p',p''\in \cM$ with $p'\neq p$ and $\lambda \in (0,1)$ with $p = \lambda p' + (1-\lambda)p''$. Evidently, $p'$ and $p''$ are both absolutely continuous w.r.t.\ $p$. To conclude, note that
    \begin{equation*}
        \frac{\mathrm{d}p'}{\mathrm{d}p} = \frac{1}{\lambda} - \frac{1-\lambda}{\lambda} \frac{\mathrm{d}p''}{\mathrm{d}p} \leq \frac{1}{\lambda} \quad p\text{-a.s.}
    \end{equation*}

\noindent
\ref{it:prop:extreme}$\,\Rightarrow\,$\ref{it:prop:bounded}:\
    Let $p'\in \cM\setminus\{p\}$ with $p'\ll p$ and $\frac{\d p'}{\d p}\leq C$ $p$-a.s.\ for some constant $C\geq 1$, and set $\eps:=\frac1C$.\ Since $\cM$ is convex, $p+\eps(p'-p)=(1-\eps)p+\eps p'\in\cM$.\ Moreover, for any $A\in\cF$,
    \[
    \big(p-\eps(p'-p)\big)(A)=\int_A (1+\eps)-\eps\frac{\d p'}{\d p}\,\d p\geq \big(1+\eps-\eps C\big)\,p(A)=\eps\, p(A)\geq 0,
    \]
    which shows that $p-\eps(p'-p)\in\cP$ and therefore $p-\eps(p'-p)\in\cM$.\ Since
    $p=\frac12\big(p+\eps(p'-p)\big)+\frac12\big(p-\eps(p'-p)\big)$
    and $p'\neq p$, it follows that $p$ is not an extreme point of $\cM$.
\end{proof}

Even though extremality cannot be fully characterized through $\ll$-minimality, $\ll$-minimal elements represent a special class of extreme points which allow a characterization in terms of suitable divergences.
Recall that, for a convex function $\varphi \colon [0,\infty) \to \R$ with $\varphi(1)=0$ and two probability measures $p,q \in \cP$ with $p \ll q$, the $\phi$-divergence of $p$ with respect to $q$, introduced by \citet{zbMATH03322635,zbMATH03252891,zbMATH03252892}, is defined as
    \begin{equation} \label{eq:divergence}
        \Div{\varphi}{p}{q} := \int_{\Omega} \varphi \left( \frac{\d p}{\d q} \right) \, \d q.
    \end{equation}
The convexity of $\varphi$ together with $\varphi(1)=0$ and Jensen's inequality implies that
\[
\Div{\varphi}{p}{q}\geq \varphi(1)=0.
\]
If $\varphi$ is strictly convex, then $\Div{\varphi}{p}{q}=0$ if and only if $p=q$.\
Moreover, convexity of $\varphi$ also implies that the difference quotient
\[
[-t,\infty)\setminus\{0\}\to \R,\quad s\mapsto\frac{\varphi(t+s)-\varphi(t)}{s}
\]
is non-decreasing for all $t > 0$, so that the right and left derivatives
\[
\varphi'_+(t):=\lim_{s\downarrow 0}\frac{\varphi(t+s)-\varphi(t)}{s}\quad \text{and}\quad \varphi'_-(t):=\lim_{s\uparrow 0}\frac{\varphi(t+s)-\varphi(t)}{s}
\]
exist in $\R$ for all $t>0$. On $(0,\infty)$, both are non-decreasing and satisfy $$\varphi_+'(0)\leq \varphi'_-(t)\leq\varphi'_+(t)\quad \text{for all }t>0,$$ where
\[
 \varphi'_+(0):=\lim_{s\downarrow 0}\frac{\varphi(s)-\varphi(0)}{s}\in [-\infty,\infty).
\]
Many familiar measures of discrepancy arise from particular choices of $\varphi$. For example,
\[
\varphi(t)=t\log t,\quad \text{for }t> 0,
\]
together with $\varphi(0)=0$, gives the Kullback--Leibler divergence,
\[
\Div{{\rm KL}}{p}{q}
=\int_{\Omega} \log\bigg(\frac{{\rm d}p}{{\rm d}q}\bigg)\,{\rm d}p,
\]
while
\[
\varphi(t)=\frac{1}{2}|t-1|,\quad \text{for }t\geq 0,
\]
gives the total variation distance. The choice
\begin{equation}\label{eq:hellinger}
\varphi(t)=\frac12\big(\sqrt{t}-1\big)^2,\quad \text{for }t\geq 0,
\end{equation}
produces the squared Hellinger distance, and
\[
\varphi(t)=(t-1)^2,\quad \text{for }t\geq 0,
\]
leads to the Pearson $\chi^2$-divergence.\ 

We start with the following auxiliary result which we will use extensively in the subsequent discussion.
\begin{lemma} \label{lem:derivative:div}
   Let $p,p',q\in\cP$ with $p\ll q$ and $p'\ll q$, and let $\varphi\colon[0,\infty)\to\R$ be a convex function satisfying
    \begin{equation}\label{eq:phi:integrable}
        \varphi(1)=0,\quad \Div{\varphi}{p}{q}<\infty,\quad\text{and}\quad \Div{\varphi}{p'}{q}<\infty.
    \end{equation}
    Then, the following statements hold.
    \begin{enumerate}
        \item The functions $\varphi\big(\frac{\d p}{\d q}\big)$ and $\varphi\big(\frac{\d p'}{\d q}\big)$ are $q$-integrable on every $A\in\cF$.\label{it:lem:integrable}
        \item Define
        \begin{equation} \label{eq:lem:g}
        g:=\begin{cases}
        \varphi'_+\big(\tfrac{\d p}{\d q}\big)\big(\tfrac{\d p'}{\d q}-\tfrac{\d p}{\d q}\big) & \text{if } \tfrac{\d p'}{\d q}\geq \tfrac{\d p}{\d q},\\[0.5em]
        \varphi'_-\big(\tfrac{\d p}{\d q}\big)\big(\tfrac{\d p'}{\d q}-\tfrac{\d p}{\d q}\big) & \text{otherwise,}
        \end{cases}
        \end{equation}
        with the convention $(-\infty)\cdot 0:=0$.\ Then, $g$ is $q$-a.s.\ well-defined with values in $[-\infty,\infty)$ and, for every $A\in\cF$,
        \begin{equation}\label{eq:lem:derivative}
        \lim_{\lambda\uparrow 1}\int_A  \frac{\varphi\big(\lambda \frac{\d p}{\d q} + ( 1- \lambda) \frac{\d p'}{\d q}\big)-\varphi\big(\frac{\d p}{\d q}\big)}{1-\lambda}\,\d q=\int_A g\,\d q\in[-\infty,\infty).
        \end{equation}
        In particular, choosing $A=\Omega$,
        \[
       \lim_{\lambda\uparrow 1}\frac{\Div{\varphi}{\lambda p+(1-\lambda)p'}{q}-\Div{\varphi}{p}{q}}{1-\lambda}=\int_\Omega g\,\d q.
        \]\label{it:lem:derivative}
    \end{enumerate}
\end{lemma}

\begin{proof}
    \ref{it:lem:integrable}:\ Convexity of $\phi$ and $\phi(1)=0$ give $\varphi\big(\frac{\d p'}{\d q}\big)\geq\varphi'_+(1)\big(\frac{\d p'}{\d q}-1\big)$ $q$-a.s., where the right-hand side is $q$-integrable.\ Hence, $\Div{\varphi}{p'}{q}<\infty$ implies that $\varphi\big(\frac{\d p'}{\d q}\big)$ is $q$-integrable on every $A \in \cF$.\ The same argument applies to $p$.

    \noindent\ref{it:lem:derivative}:\ For $h\in(0,1]$, consider the difference quotient
    \[
    g_h:=\frac{\varphi\Big(\frac{\d p}{\d q}+h\big(\frac{\d p'}{\d q}-\frac{\d p}{\d q}\big)\Big)-\varphi\big(\frac{\d p}{\d q}\big)}{h}.
    \]
    Setting $h:=1-\lambda$ with $\lambda\in [0,1)$, the integrand in \eqref{eq:lem:derivative} equals $g_{1-\lambda}$.\ By convexity of $\varphi$, the map $h\mapsto g_h$ is non-decreasing and, therefore, $\lim_{h \downarrow 0} g_h = g$ $q$-a.s.\ Moreover,
    \[
    g_h\leq g_1=\varphi\bigg(\frac{\d p'}{\d q}\bigg)-\varphi\bigg(\frac{\d p}{\d q}\bigg),
    \]
    which is $q$-integrable by \ref{it:lem:integrable}.\ Applying monotone convergence to the non-negative family $(g_1-g_h)_{h\in(0,1]}$, which is non-decreasing as $h\downarrow 0$, gives
    \[
    \lim_{h\downarrow 0}\int_A(g_1-g_h)\,\d q=\int_A(g_1-g)\,\d q.
    \]
    After subtracting $\int_Ag_1\,\d q$, which is finite by \ref{it:lem:integrable}, we obtain \eqref{eq:lem:derivative}.
\end{proof}
The following theorem shows that $\ll$-minimality is equivalent to the blow-up of a divergence map along feasible segments, for every divergence associated with a convex function $\varphi$ with infinite negative slope at zero and with respect to every dominating measure.

\begin{theorem}\label{thm:minimal:divergence}
    For $p\in \cM$, the following statements are equivalent.
    \begin{enumerate}
        \item $p$ is $\ll$-minimal in $\cM$.\label{it:minimal}
        \item For every $p'\in\cM\setminus\{p\}$, every $q\in\cM$ with $p\ll q$, $p'\ll q$, and every convex function $\varphi\colon[0,\infty)\to\R$ satisfying 
        \begin{equation}\label{eq:phi:blowup}
        \varphi(1)=0,\quad \varphi'_+(0)=-\infty,\quad \Div{\varphi}{p}{q}<\infty,\quad \Div{\varphi}{p'}{q}<\infty,
    \end{equation}
    one has
    \begin{equation}\label{eq:blowup}
    \lim_{\lambda\uparrow 1}\frac{\Div{\varphi}{\lambda p+(1-\lambda)p'}{q}-\Div{\varphi}{p}{q}}{1-\lambda}=-\infty.
    \end{equation}
    In particular, $\lambda=1$ is a strict local maximizer of the map
    \begin{equation*}
        [0,1]\to[0,\infty),\quad \lambda\mapsto\Div{\varphi}{\lambda p+(1-\lambda)p'}{q}.
    \end{equation*}\label{it:minimal:blowup}
    \end{enumerate}
\end{theorem}
\begin{proof}
    \noindent\ref{it:minimal}$\,\Rightarrow\,$\ref{it:minimal:blowup}:\ Let $p'\in\cM\setminus\{p\}$, let $q\in\cM$ with $p\ll q$ and $p'\ll q$.\ Moreover, let $\varphi$ satisfy \eqref{eq:phi:blowup}.\ Since $p$ is $\ll$-minimal, $p'\not\ll p$, so there exists $A_0\in\cF$ with $p(A_0)=0$ and $p'(A_0)>0$.\
    From $$p(A_0)=\int_{A_0}\frac{\d p}{\d q}\,\d q=0,$$ 
    it follows that $\frac{\d p}{\d q}=0$ $q$-a.s.\ on $A_0$.
    Set $f(\lambda):=\Div{\varphi}{\lambda p+(1-\lambda)p'}{q}$, and write
    \begin{align*}
    \frac{f(\lambda)-f(1)}{1-\lambda} & = \int_{A_0} \frac{\varphi\big(\lambda\frac{\d p}{\d q}+(1-\lambda)\frac{\d p'}{\d q}\big)-\varphi\big(\frac{\d p}{\d q}\big)}{1-\lambda}\,\d q\\
    & \qquad +\int_{A_0^{\rm c}}\frac{\varphi\big(\lambda\frac{\d p}{\d q}+(1-\lambda)\frac{\d p'}{\d q}\big)-\varphi\big(\frac{\d p}{\d q}\big)}{1-\lambda}\,\d q.
    \end{align*}
    By Lemma \ref{lem:derivative:div} and the definition of $A_0$, we obtain
    \begin{align*}
    \lim_{\lambda \uparrow 1} &\int_{A_0} \frac{\varphi\big(\lambda\frac{\d p}{\d q}+(1-\lambda)\frac{\d p'}{\d q}\big)-\varphi\big(\frac{\d p}{\d q}\big)}{1-\lambda}\,\d q \\
    &\qquad= \int_{A_0} \phi_+'(0)\cdot \frac{\d p'}{\d q}\,\d q= \phi_+'(0)\cdot p'(A_0)=-\infty.
    \end{align*}
    For the other integral, we use convexity of $\phi$ to obtain
    \[
    \lim_{\lambda \uparrow 1} \int_{A_0^{\rm c}} \frac{\varphi\big(\lambda\frac{\d p}{\d q}+(1-\lambda)\frac{\d p'}{\d q}\big)-\varphi\big(\frac{\d p}{\d q}\big)}{1-\lambda}\,\d q \le \int_{A_0^{\rm c}} \varphi\bigg(\frac{\d p'}{\d q}\bigg)-\varphi\bigg(\frac{\d p}{\d q}\bigg)\,\d q < \infty,
    \]
    where the boundedness also follows from Lemma \ref{lem:derivative:div}\ref{it:lem:integrable}.

    Combining the two pieces, $\lim_{\lambda\uparrow 1}\frac{f(\lambda)-f(1)}{1-\lambda}=-\infty$, which implies that $\lambda=1$ is a strict local maximizer of $f$.

    \noindent\ref{it:minimal:blowup}$\,\Rightarrow\,$\ref{it:minimal}:\ Suppose that $p$ is not $\ll$-minimal in $\cM$, and let $p'\in\cM\setminus\{p\}$ with $p'\ll p$.\ Take $q:=p$, so that $p\ll q$ and $p'\ll q$.\ Let $\varphi$ be convex with $\varphi(1)=0$, $\varphi'_+(0)=-\infty$, and of at most linear growth, i.e., $\varphi(t)\leq a+bt$ for all $t\geq 0$ and some $a,b\in\R$, as is the case, e.g., for the squared Hellinger distance~\eqref{eq:hellinger}.\ Then $\Div{\varphi}{p}{q}=0$ and $\Div{\varphi}{p'}{q}\leq a+b\int_\Omega \frac{\d p'}{\d q}\,\d q<\infty$, so that $\varphi$ satisfies~\eqref{eq:phi:blowup}.\ Setting again $f(\lambda):=\Div{\varphi}{\lambda p+(1-\lambda)p'}{q}$ and using that $\frac{\d p}{\d q}=1$ $q$-a.s., by Jensen's inequality, we obtain 
    \[
    f(\lambda)-f(1)=f(\lambda)\geq 0 \quad \text{for all }\lambda\in [0,1),
    \]
    which contradicts \ref{it:minimal:blowup}.
\end{proof}

Note that a suitable dominating measure for the previous theorem is always $q = \tfrac{1}{2}(p+p')$.\ In this case, $\frac{\d p}{\d q},\frac{\d p'}{\d q} \le 2$ $q$-a.s., so that condition \eqref{eq:phi:blowup} holds for every convex function $\varphi$ with $\varphi(1)=0$ and $\varphi'_+(0)=-\infty$.
Furthermore, the condition $\varphi'_+(0)=-\infty$ in Theorem~\ref{thm:minimal:divergence} is essential and cannot be omitted, even if $\Omega$ is finite and $\varphi$ is strictly convex.
\begin{example}\label{ex:finite:slope}
Let $\Omega=\{1,2,3,4\}$, let $g=(0,2,4,0)$, and consider the generalized moment set
\[
\cM:=\bigg\{r\in\cP(\Omega)\,\bigg|\,\sum_{i=1}^4 g_i r_i=1\bigg\}.
\]
Define
\[
p=\bigg(\frac12,\frac12,0,0\bigg),\quad
p'=\bigg(\frac34,0,\frac14,0\bigg),\quad\text{and}\quad
q=\bigg(\frac1{16},\frac{1}6,\frac16,\frac{29}{48}\bigg).
\]
All three measures belong to $\cM$, and $q$ has full support.\ Moreover, the probability measure $p$ is $\ll$-minimal in $\cM$.\ Indeed, if $r\in\cM$ and $r\ll p$, then $r$ is supported on $\{1,2\}$, and the equations $r_1+r_2=1$ and $2r_2=1$ force $r=p$.

Take the convex function $\varphi(t)=(t-1)^2$ associated with the Pearson $\chi^2$-divergence, for which $\varphi(1)=0$ and $\varphi'_+(0)=-2$.\ For
\[
f(\lambda):=\Div{\varphi}{\lambda  p+(1-\lambda) p'}{q},
\]
differentiation of $\phi$, together with Lemma \ref{lem:derivative:div}, give
\begin{align*}
\lim_{\lambda\uparrow 1} \frac{f(\lambda)-f(1)}{1-\lambda}
&=2\sum_{i=1}^4 q_i\bigg(1-\frac{p_i}{q_i}\bigg)\frac{p_i-p'_i}{q_i}
 =2\sum_{i=1}^4\frac{p_i(p_i'-p_i)}{q_i}\\
&=\frac{p_1'-p_1}{q_1}+\frac{p_2'-p_2}{q_2}=16\times \frac{1}4-6\times\frac12
=1>0.
\end{align*}
Consequently $f(\lambda)>f(1)$ for every $\lambda\in (0,1)$ sufficiently close to $1$. Thus, $\lambda=1$ is a strict local minimizer, rather than a local maximizer.\ This shows that the hypothesis of infinite negative slope at zero in Theorem \ref{thm:minimal:divergence} cannot be omitted.
\end{example}

\begin{example}\label{ex:blowup:divergences}
    The following prominent $\varphi$-divergences satisfy the properties stated in Theorem \ref{thm:minimal:divergence}:\ the $\alpha$-divergence for $\alpha \in (0,1]$, the Kullback-Leibler divergence, the Jensen-Shannon divergence, and the squared Hellinger distance.
\end{example}

In view of Theorem \ref{thm:minimal:divergence}, it is natural to ask whether the blow-up of a divergence map also characterizes the extreme points of~$\cM$. Revisiting Example \ref{ex:extreme:not:minimal}, we observe that this is not true in general.

\begin{example}\label{ex:extreme:no:blowup}
    In the setting of Example~\ref{ex:extreme:not:minimal}, consider the extreme point $p=p_1$, which is not $\ll$-minimal in $\cM$.

    Choose $p'=q=\mu\in\cM$, so that $\Div{\varphi}{p'}{q}=0<\infty$ for all $\phi$ with $\phi(1)=0$, and let $\varphi\colon [0,\infty)\to \R$ be convex and differentiable with $\varphi(1)=0$, $\varphi'_+(0)=-\infty$, and $\Div{\varphi}{p}{q}=\int_0^1 \varphi(2x)\, \d x<\infty$.\ The latter is, for example, satisfied if $\varphi(t)=t\log t$ for $t>0$, since, in this case, $\Div{\varphi}{p}{q}=\log 2-\tfrac12$.\ By Lemma~\ref{lem:derivative:div},
    \begin{align*}
    \lim_{\lambda\uparrow1}&\frac{\Div{\varphi}{\lambda p+(1-\lambda)p'}{q}-\Div{\varphi}{p}{q}}{1-\lambda}\\
    &\qquad =\int_0^1\varphi'(2x)(1-2x)\,\d x  =\frac12\int_0^2\varphi'(u)(1-u)\,\d u,
    \end{align*}
    and integration by parts yields
    \begin{align*}
    \frac12\int_0^2\varphi'(u)(1-u)\,\d u
    & =\frac12\bigg(\int_0^2\varphi(u)\,\d u-\varphi(0)-\varphi(2)\bigg)\\
    &=\Div{\varphi}{p}{q}-\frac{\varphi(0)+\varphi(2)}{2}\in \R.
    \end{align*}
    Moreover, using the convexity of $\varphi$ together with $\varphi(1)=0$,
    \[
    \frac12\int_0^2\varphi'(u)(1-u)\,\d u \leq \frac12 \int_0^2\varphi(1)-\varphi(u)\,\d u=-\frac12\int_0^2\varphi(u)\, \d u =-\Div{\varphi}{p}{q}.
    \]
   The differential ratio is therefore negative but finite in the limit.
\end{example}

A characterization of the extreme points of $\cM$ as strict local maximizers of a divergence map can be recovered after suitably restricting both the class of dominating measures and the class of functions $\varphi$ under consideration.\
To that end, we introduce the notion of an \textit{admissible dominating measure}, which plays a central role in the subsequent discussion.

\begin{definition}\label{def:admissible}
    Let $p,p'\in\cM$ with $p\neq p'$.\ We say that $q$ is a \textit{$p'$-admissible dominating measure} of $p$ if $q\in\cM$, $p \ll q$ and the following two conditions are satisfied
    \begin{enumerate}[label=\textup{(\alph*)}]
        \item $p'\ll q$ and $\frac{\d p'}{\d q}$ is $q$-a.s.\ bounded,\label{it:admissible:bounded}
        \item there exists $\eps_0 \in (0,1]$ such that $\frac{\d p'}{\d q}>\frac{\d p}{\d q}$ $q$-a.s.\ on $\big\{\frac{\d p}{\d q}<\eps_0\big\}$, i.e.,
        \[
         q\bigg(\bigg\{\frac{\d p'}{\d q}\leq \frac{\d p}{\d q}<\eps_0\bigg\}\bigg)=0.
        \]
        \label{it:admissible:eps}
    \end{enumerate}
\end{definition}

\begin{remark}\label{rem:admissible}
    For every pair $p,p'\in\cM$ with $p\neq p'$, there exists a $p'$-admissible dominating measure of $p$.\ A canonical choice is
    \[
    q:=\tfrac{1}{2}(p+p')\in\cM,
    \]
    for which $p\ll q $ and $p'\ll q$ with $\frac{\d p}{\d q}+\frac{\d p'}{\d q}=2$ $q$-a.s., and in particular $\frac{\d p'}{\d q}\leq 2$ $q$-a.s., so that condition \ref{it:admissible:bounded} in Definition \ref{def:admissible} is satisfied.\ Moreover, condition \ref{it:admissible:eps} holds with $\eps_0=1$.\ Indeed, on the set $\big\{\frac{\d p}{\d q}<1\big\}$, one has
    \[
    \frac{\d p'}{\d q}=2-\frac{\d p}{\d q}>1>\frac{\d p}{\d q}\quad q\text{-a.s.}
    \]
    Hence $q=\tfrac{1}{2}(p+p')$ is a $p'$-admissible dominating measure of $p$.\ Admissibility is, however, a genuine restriction on $q$. For instance, any $q\in \cP$ with
    \[
     q\bigg(\bigg\{\frac{\d p'}{\d q}=0=\frac{\d p}{\d q}\bigg\}\bigg)>0
    \] 
    violates the strict inequality in condition \ref{it:admissible:eps}.
\end{remark}

The previous discussion leads to the following main result, which provides several alternative characterizations of the extreme points of $\cM$.
\begin{theorem}\label{thm:extreme:divergence}
   For an element $p \in \cM$, the following statements are equivalent.
    \begin{enumerate}
        \item $p$ is an extreme point of $\cM$.\label{it:thm:extreme}
        \item \label{it:thm:bounded} There exists no $p' \in \cM\setminus \{p\}$ with $p'\ll p$ and $p$-a.s.\ bounded density $\frac{\mathrm{d}p'}{\mathrm{d}p}$.
        \item \label{it:thm:perturbation} For all $q\in \cM$ with $p\ll q$ and all $\eps\in (0,1)$, there exists no $p' \in \cM\setminus\{p\}$ with $p'\ll q$ and $q$-a.s.\ bounded density $\frac{\mathrm{d}p'}{\mathrm{d}q}$ and
        \begin{equation}\label{eq:perturbation}
           (1+\eps) \frac{\mathrm{d}p}{\mathrm{d}q} -\eps\frac{\mathrm{d}p'}{\mathrm{d}q}\geq 0 \quad q\text{-a.s.}
        \end{equation}
        \item \label{it:thm:maximizer}
        For every $p'\in\cM\setminus\{p\}$, every $p'$-admissible dominating measure $q$ of $p$, every $\eps_0 \in (0,1]$ satisfying condition \ref{it:admissible:eps} in Definition~\ref{def:admissible}, and every continuous convex function $\varphi \colon [0,\infty)\to [0,\infty)$ with $\varphi(0)>0$ and $\varphi=0$ on $[a,\infty)$ for some $a\in(0,\eps_0)$, the point $\lambda = 1$ is a strict local maximizer of the map
        \begin{equation}\label{eq:divergence:map}
        [0,1] \to [0,\infty),\quad \lambda \mapsto \Div{\varphi}{\lambda p + (1-\lambda)p'}{q}.
        \end{equation}
    \end{enumerate}
\end{theorem}

\begin{proof}
The equivalence \ref{it:thm:extreme}$\,\Leftrightarrow\,$\ref{it:thm:bounded} follows from Proposition \ref{prop:extreme:minimal}.

\noindent
\ref{it:thm:extreme}$\,\Rightarrow\,$\ref{it:thm:perturbation}:\
    Let $q \in \cM$ with $p\ll q$, $\eps\in (0,1)$, and $p'\in \cM\setminus\{p\}$ with $p'\ll q$ such that \eqref{eq:perturbation} holds. Since $\cM$ is convex, $p + \epsilon (p'-p) = (1-\epsilon) p + \epsilon p' \in \cM$. Now, for any $A \in \cF$,
    \[
        (p-\epsilon (p'-p))(A) = \int_A (1+\eps) \frac{\mathrm{d}p}{\mathrm{d}q}-\eps \frac{\mathrm{d}p'}{\mathrm{d}q}\, \d q\geq 0,
    \]
    which shows that $p-\epsilon (p'-p)\in \cP$ and therefore $p-\epsilon (p'-p)\in \cM$. Since $p=\frac12\big(p+\epsilon (p'-p)\big)+\frac12\big(p-\epsilon (p'-p)\big)$, $p$ is not an extreme point of $\cM$.

\noindent
\ref{it:thm:perturbation}$\,\Rightarrow\,$\ref{it:thm:maximizer}:\
    Assume that \ref{it:thm:perturbation} is satisfied.\ Let $p' \in \cM\setminus\{p\}$, $q$ be a $p'$-admissible dominating measure of $p$, $\eps_0 \in (0,1]$ satisfy \ref{it:admissible:eps}, $\varphi\colon[0,\infty)\to[0,\infty)$ be a continuous convex function with $\varphi(0)>0$ and $\varphi=0$ on $[a,\infty)$ for some $a\in(0,\eps_0)$, and $C\geq 1$ such that $\frac{\d p'}{\d q}\leq C$ $q$-a.s.

    We first observe that, by convexity, $\varphi$ attains its maximum over $[0,a]$ at an endpoint, and since $\varphi(a)=0<\varphi(0)$, we get $0\leq\varphi\leq\varphi(0)$, so that $\Div{\varphi}{p}{q}\leq\varphi(0)<\infty$.\ Without loss of generality, we may further assume that $\varphi>0$ on $[0,a)$, otherwise, replace $a$ by $\tilde a:=\sup\{t\geq 0\,\vert\,\varphi(t)>0\}\leq a$.\
    Since $\varphi$ is convex with $\varphi=0$ on $[a,\infty)$, we have $\varphi'_+=\varphi'_-=0$ on $(a,\infty)$ and $\varphi'_+(a)=0$. Furthermore, $\varphi'_+<0$ on $[0,a)$.\ As a matter of fact, if $\varphi'_+(t)\geq 0$ for some $t\in[0,a)$, then $\varphi$ would be non-decreasing on $[t,a]$, so that $0<\varphi(t)\leq\varphi(a)=0$, a contradiction.

    It remains to show that $\lambda=1$ is a strict local maximizer of the map $f \colon [0,1] \to \R$, $\lambda \mapsto \Div{\varphi}{\lambda p + (1 - \lambda) p'}{q}$.\
    By Lemma~\ref{lem:derivative:div}\ref{it:lem:derivative}, it holds
    \[
    \lim_{\lambda\uparrow 1}\frac{f(\lambda)-f(1)}{1-\lambda}=\int_\Omega g\,\d q
    \]
    with $g$ defined as in \eqref{eq:lem:g}.\ We start by analyzing the function $g$.\ On the event $\big\{\frac{\d p}{\d q}>a\big\}$, both one-sided derivatives of $\varphi$ vanish, so $g=0$.\
    Since $$\bigg\{\frac{\d p}{\d q}\leq a\bigg\}\subseteq\bigg\{\frac{\d p}{\d q}<\eps_0\bigg\},$$ hypothesis \ref{it:admissible:eps} gives $\frac{\d p'}{\d q}>\frac{\d p}{\d q}$ $q$-a.s.\ on $\big\{\frac{\d p}{\d q}=a\big\}$.\ Together with $\varphi'_+(a)=0$, this implies that $g=0$ on the event $\big\{\frac{\d p}{\d q}=a\big\}$, so that
    \[
    \lim_{\lambda\uparrow 1}\frac{f(\lambda)-f(1)}{1-\lambda}=\int_{\{\frac{\d p}{\d q}<a\}}\varphi'_+\!\bigg(\frac{\d p}{\d q}\bigg)\bigg(\frac{\d p'}{\d q}-\frac{\d p}{\d q}\bigg)\,\d q.
    \]
    Next, we show that $q\big(\big\{\frac{\d p}{\d q}<a\big\}\big)>0$.\ To that end, choose $\eps\in (0,1)$ such that $\frac{\eps C}{1+\eps}<a$. By \ref{it:thm:perturbation} applied to this $\eps$, the set
    \[
    A:=\bigg\{(1+\eps)\frac{\d p}{\d q}<\eps \frac{\d p'}{\d q}\bigg\}
    \]
    satisfies $q(A)>0$, and $A\subseteq \big\{\frac{\d p}{\d q}<\frac{\eps C}{1+\eps}\big\}\subseteq \big\{\frac{\d p}{\d q}<a\big\}$. Finally, by hypothesis \ref{it:admissible:eps}, $\varphi'_+\!\big(\frac{\d p}{\d q}\big)\big(\frac{\d p'}{\d q}-\frac{\d p}{\d q}\big)<0$ $q$-a.s.\ on $\big\{\frac{\d p}{\d q}<a\big\}$, and this set has positive $q$-measure, which implies that
    \[
    \lim_{\lambda\uparrow 1}\frac{f(\lambda)-f(1)}{1-\lambda}<0.
    \]
    Consequently, $\frac{f(\lambda)-f(1)}{1-\lambda}$ is strictly negative for $\lambda\in [0,1)$ sufficiently close to $1$, so that $\lambda=1$ is a strict local maximizer of the map \eqref{eq:divergence:map}.

\noindent
\ref{it:thm:maximizer}$\,\Rightarrow\,$\ref{it:thm:bounded}:\
    Suppose there exists $p' \in \cM\setminus \{p\}$ with $p'\ll p$ and $p$-a.s.\ bounded density $\frac{\mathrm{d}p'}{\mathrm{d}p}$.\
    Then, choosing $q=p$, the hypotheses on $q$ in \ref{it:thm:maximizer} are trivially satisfied. Moreover, every convex function $\varphi$ as in \ref{it:thm:maximizer} satisfies $\varphi(1)=0$, since $\varphi=0$ on $[a,\infty)$ with $a<\eps_0:=1$, and therefore, by Jensen's inequality,
    \[
    \Div{\varphi}{p}{q}=\Div{\varphi}{p}{p}=0\leq \Div{\varphi}{\lambda p+(1-\lambda)p'}{q}.
    \]
    Hence, $\lambda=1$ is a global minimizer of the map \eqref{eq:divergence:map} and therefore not a strict local maximizer, contradicting \ref{it:thm:maximizer}.
\end{proof}

\begin{remark}\label{rem:admissible:role}
We briefly comment on item \ref{it:thm:maximizer} in Theorem~\ref{thm:extreme:divergence}. The $p'$-admissible dominating measure $q$ of $p$ plays the role of a reference measure against which $p$ and $p'$ are read off as densities.\ Condition \ref{it:admissible:eps} in Definition~\ref{def:admissible} then constrains how these densities relate to each other.\ In particular, $p'$ is required to be strictly heavier than $p$, in the $q$-density sense, on a region where $p$ is light with respect to the dominating measure $q$. Note that no uniform separation between $\frac{\d p'}{\d q}$ and $\frac{\d p}{\d q}$ is assumed.\ The density $\frac{\d p'}{\d q}$ may itself approach $0$, as long as it stays above $\frac{\d p}{\d q}$ $q$-a.s.\ on this region.
\end{remark}

\section{Examples and applications}\label{sec:applications}
In this section, we apply our results to several settings.\ We revisit classical results and investigate extensions and novel directions.

\subsection{Finite spaces}\label{sec:finite}
In the case of a finite space $\Omega$, $\ll$-minimality is equivalent to extremality. Indeed, on a finite space, every density is bounded, so that $\ll$-minimality coincides with condition \ref{it:thm:bounded} in Theorem~\ref{thm:extreme:divergence}.\ This roughly corresponds to the characterization of vertices of polytopes of the form $\{x\in\R^\Omega \,\vert\, x\geq 0,\ Ax=b\}$ in \citet[Chapter 8]{schrijver1986theory}.\ In combination with Theorem~\ref{thm:minimal:divergence}, this yields the following characterization of extreme points in terms of strict local maximizers of any divergence.\ 

\begin{theorem}\label{thm:finite:divergence}
    Let $\Omega$ be finite and $\cF$ be the power set.\
    \begin{enumerate}
    \item[a)]\label{it:finite:measure} There exists a measure $q\in \cM$ with $p'\ll q$ for all $p'\in \cM$.
    \item[b)]\label{it:finite:equivalances} Let $q\in \cM$ with $p'\ll q$ for all $p'\in \cM$ and $\varphi\colon[0,\infty)\to\R$ be a convex function with $\varphi(1)=0$ and $\varphi_+'(0)=-\infty$.\ Then, for $p\in \cM$, the following statements are equivalent.
    \begin{enumerate}
        \item\label{it:finite:extreme}$p$ is an extreme point of $\cM$.
        \item\label{it:finite:blowup} For every $p'\in \cM\setminus\{p\}$, it holds
    \begin{equation}
            \lim_{\lambda\uparrow 1}\frac{\Div{\varphi}{\lambda p+(1-\lambda)p'}{q}-\Div{\varphi}{p}{q}}{1-\lambda}=-\infty.\label{eq:finite:blowup}
    \end{equation} 
    \item\label{it:finite:maximizer} $p$ is a strict local maximizer of the function
        \[
         \cM\to \R,\quad p'\mapsto \Div{\varphi}{p'}{q}.
        \]
    \end{enumerate}
    \end{enumerate}
\end{theorem}

\begin{proof}\
\begin{enumerate}
\item[a)] For each
\[
\omega_0\in \Omega_0:=\big\{\omega\in \Omega\,\big|\, \exists p'\in \cM \colon p'(\{\omega\})>0\big\},
\]
let $p_{\omega_0}\in \cM$ with $p_{\omega_0}(\{\omega_0\})>0$.\ Since $\Omega$ is finite and $\cM$ is nonempty, it follows that $|\Omega_0|\in \N$.\ Hence,
\[
q:=\frac1{|\Omega_0|}\sum_{\omega_0\in \Omega_0} p_{\omega_0}
\]
satisfies $p'\ll q$ for all $p'\in \cM$.\ 
\item[b)] Let $q\in \cM$ with $p'\ll q$ for all $p'\in \cM$.\ Since $\Omega$ is finite, it follows that
\[
\Div{\varphi}{p'}{q}<\infty
\]
for all $p'\in \cM$. We now prove the equivalence.

\ref{it:finite:extreme}$\,\Rightarrow\,$\ref{it:finite:blowup}:\ Assume $p$ is an extreme point of $\cM$.\ As observed above, by Theorem~\ref{thm:extreme:divergence}, it follows that $p$ is $\ll$-minimal in $\cM$.\ Let $p'\in\cM\setminus\{p\}$.\ Since $\Omega$ is finite, we have $\Div{\varphi}{p}{q}<\infty$ and $\Div{\varphi}{p'}{q}<\infty$, as observed above, so that condition \eqref{eq:phi:blowup} is met.\ Theorem~\ref{thm:minimal:divergence} then yields the limit~\eqref{eq:finite:blowup}, and hence the strict local maximality.

    \noindent \ref{it:finite:blowup}$\,\Rightarrow\,$\ref{it:finite:maximizer}:\
Since $\cM = H \cap \cP$ and $H$ is an affine subspace, $\cM$ has only finitely many distinct extreme points $p_1', \ldots, p_m' \in \cM$ with $m\in \N$, cf.\ \citet[Chapter 3.1, Statement 4]{grunbaum2003convex}. If $\cM$ is a singleton, there is nothing to show. Since every element of $\cM$ is the convex combination of $p_1', \ldots, p_m'$ by the Krein--Milman theorem, we may thus assume $m \geq 2$ in the following.
For each $i \in \{1, \ldots, m\}$ with $p_i' \neq p$, we find $\lambda_i \in (0,1)$ such that $\Div{\varphi}{\lambda p + (1-\lambda) p_i'}{q}$ is strictly increasing for $\lambda \in [\lambda_i, 1]$ by \ref{it:finite:blowup}. Let \[
\lambda^*:=\max\big\{\lambda_i\, \big|\, i\in\{1,\ldots, m\}\text{ with }p_i'\neq p \big\} \in (0,1)
\]be their maximum, so that $\Div{\varphi}{\lambda^* p + (1-\lambda^*)p_i'}{q} < \Div{\varphi}{p}{q}$ for all $i=1,\ldots, m$ with $p_i' \neq p$.\ Consider the set $U := \lambda^* p + (1-\lambda^*) \cM$ and note that it contains a neighborhood of $p$ that is open relative to $\cM$. Let $p'\in U\setminus \{p\}$, and use the Krein--Milman theorem to write $$p' = \lambda^* p + (1-\lambda^*) \sum_{i=1}^m \theta_i p_i' \in U$$ with $\theta_i \geq 0$ for $i=1,\ldots, m$ and $\sum_{i=1}^m \theta_i = 1$. Since $p'\neq p$, there is at least one $i\in \{1,\ldots, m\}$ with $\theta_i>0$ and $p_i' \neq p$.\ By convexity of $\varphi$, we thus have
\begin{align*}
    \Div{\varphi}{p'}{q} &= \Div{\varphi}{\sum_{i=1}^m \theta_i (\lambda^* p + (1-\lambda^*) p_i')}{q}\\
    &\leq \sum_{i=1}^m \theta_i \Div{\varphi}{\lambda^* p + (1-\lambda^*)p_i'}{q} < \Div{\varphi}{p}{q}.
\end{align*}
Consequently, $p$ is a strict local maximizer of $\Div{\varphi}{\,\cdot\,}{q}$ on $\cM$.
    
\ref{it:finite:maximizer} $\,\Rightarrow\,$\ref{it:finite:extreme}:\ Suppose that $p$ is not an extreme point of $\cM$.\ Then, there exist $p_1,p_2\in \cM\setminus\{p\}$ and $\theta\in (0,1)$ such that $p=\theta p_1+(1-\theta) p_2$.\ Let $d:=p_1-p_2$, and, for every $\delta\in\big(0,\min\{\theta,1-\theta\}\big)$, define $p_\delta^{+}:=p+\delta d$ and $p_\delta^{-}:=p-\delta d$. Note that $p=\tfrac12 p_\delta^{+}+\tfrac12 p_\delta^{-}$ and $$p_\delta^{\pm} = (\theta \pm \delta)p_1 + (1-\theta \mp \delta) p_2\in \cM.$$
    By convexity of $\varphi$,
\begin{align*}
    \Div{\varphi}{p}{q}\leq \frac12\Div{\varphi}{p_\delta^{+}}{q}+\frac12\Div{\varphi}{p_\delta^{-}}{q} \leq \max\big\{\Div{\varphi}{p_\delta^{+}}{q},\Div{\varphi}{p_\delta^{-}}{q}\big\}.
    \end{align*}
Hence, $p$ is not a strict local maximizer of $\Div{\varphi}{\,\cdot\,}{q}$ on $\cM$.
    \end{enumerate}
    The proof is complete.
\end{proof}

To the best of our knowledge, the connection between extremality and local maximization of a divergence in finite spaces was first noted in \cite{bauch2025correlationuncertaintydecisiontheoreticapproach} in the setting of multi-marginal couplings.\ Theorem~\ref{thm:finite:divergence} generalizes \cite[Theorem 4]{bauch2025correlationuncertaintydecisiontheoreticapproach} in two directions.\ First, it applies  to arbitrary sets of the form $\cM=\cP\cap H$ as described at the beginning of Section \ref{sec:extreme}, and not only to sets of multi-marginal couplings. Second, it is independent of the choice of the dominating reference measure $q$.

\subsection{Integral constraints}\label{sec:integral}
Given a family $(f_i)_{i \in \cI}$ of measurable functions $\Omega\to \R$ with a nonempty index set $\cI$, and $(a_i)_{i \in \cI} \subset \R$, we consider
\[
\cM = \bigg\{ p \in \cP \,\bigg\vert\, f_i\in L^1(p)\text{ with }\int_\Omega f_i\,\d p = a_i \text{ for all } i \in \cI \bigg\},
\]
where $L^1(p)$ denotes the set of measurable functions $f \colon \Omega \to \R$ with $\int_\Omega |f|\,\d p < \infty$.
In the case where $\cI = \{1, \dots, n\}$, for $n \in \N$, we give a complete characterization of the extreme points of $\cM$.\ 
This generalizes and extends the classical characterization by \cite{karr1983extreme} and \cite{winkler1988extreme}, which we state as a corollary below.\
Recall that a set $A\in \cF$ is a \emph{$p$-atom} if $p(A)>0$ and every $B\in \cF$ with $B\subseteq A$ satisfies $p(B)\in\{0,p(A)\}$.\ For a $p$-atom $A\in \cF$, we consider
\begin{equation}\label{eq:atom:measure}
p_A(B):=\frac{p(A\cap B)}{p(A)},\qquad B\in \cF.
\end{equation}
Then, $p_A$ is a $\{0,1\}$-valued probability measure, and therefore an extreme point of $\cP$.\ Moreover, every measurable function $\Omega\to\R$ is $p_A$-a.s.\ constant, and we denote by $f_i(A)$ the constant value that $f_i$ takes $p_A$-a.s.\ for $i=1,\ldots, n$.

\begin{theorem}\label{thm:moment:extreme}
Assume that $\cI=\{1,\ldots,n\}$ for some $n\in \N$. Then, for $p\in \cM$, the following statements are equivalent.
\begin{enumerate}
    \item $p$ is an extreme point of $\cM$.\label{it:moment:extreme}
    \item There exist $m\in \N$, coefficients $\alpha_1,\ldots, \alpha_m>0$, and pairwise disjoint $p$-atoms $A_1,\ldots, A_m\in \cF$ such that
    \[
    p=\sum_{j=1}^m\alpha_j\, p_{A_j}
    \]
    and the matrix
    \[
    \left(\begin{matrix}
    1 & \dots & 1 \\
    f_1(A_1) & \dots & f_1(A_m) \\
    \vdots & & \vdots \\
    f_n(A_1) & \dots & f_n(A_m)
    \end{matrix}\right)
    \]
    has rank $m$, i.e., the vectors
    \[
    v_j:=\big(1,f_1(A_j),\ldots,f_n(A_j)\big)^\top\in\R^{n+1},
    \qquad j=1,\ldots,m,
    \]
    are linearly independent. In particular, extremality forces $m\leq n+1$.\label{it:moment:atomic}
    \item $p$ is $\ll$-minimal in $\cM$.\label{it:moment:minimal}
    \item For every $p'\in\cM\setminus\{p\}$, every $q\in\cM$ with $p\ll q$ and $p'\ll q$, and every convex function $\varphi\colon[0,\infty)\to\R$ satisfying \eqref{eq:phi:blowup},
    one has
    \[
    \lim_{\lambda\uparrow 1}\frac{\Div{\varphi}{\lambda p+(1-\lambda)p'}{q}-\Div{\varphi}{p}{q}}{1-\lambda}=-\infty.
    \]
    In particular, $\lambda=1$ is a strict local maximizer of the map
    \begin{equation*}
        [0,1]\to[0,\infty),\quad \lambda\mapsto\Div{\varphi}{\lambda p+(1-\lambda)p'}{q}.
    \end{equation*}\label{it:moment:blowup}
\end{enumerate}
\end{theorem}

\begin{proof}
\ref{it:moment:extreme}$\,\Rightarrow\,$\ref{it:moment:atomic}:\ Suppose first that $p\in\cM$ admits $m\in \N$ pairwise disjoint measurable sets $A_1,\ldots,A_{m}\in \cF$ with $p(A_j)>0$, for $j=1,\ldots, m$, such that the vectors
\[
u_j:=\bigg(p(A_j),\int_{A_j}f_1\,\d p,\ldots,\int_{A_j}f_n\,\d p\bigg)\in\R^{n+1},\quad \text{for }j=1,\ldots,m,
\]
 are linearly dependent in $\R^{n+1}$.\ Clearly, these vectors will always be linearly dependent if $m\geq n+2$.\ Then, there exist coefficients $c_1,\ldots,c_{m}\in \R$, not all zero, with $\sum_{j=1}^{m} c_j u_j=0$.\ Then, $h:=\sum_{j=1}^{m} c_j\one_{A_j}$ is bounded by $\max\{|c_1|,\ldots,|c_{m}|\}$ and not $p$-a.s.\ equal to zero, but
\[
\int_\Omega h\,\d p=0
\quad\text{and}\quad
\int_\Omega f_i h\,\d p=0\quad\text{for }i=1,\ldots,n.
\]
For $0<\varepsilon \leq \big(\max\{|c_1|,\ldots, |c_{m}|\}\big)^{-1}$, the measures
\[
\d p_\pm:=(1\pm\varepsilon h)\,\d p
\]
are therefore distinct elements of $\cM$ and satisfy $p=\tfrac12(p_++p_-)$. Hence, $p$ is not extreme.\

Now, let $p\in \cM$ be an extreme point.\ Then, by the previous part of the proof, there exists a family $A_1,\ldots, A_m$ of pairwise disjoint sets with positive $p$-mass of maximal cardinality $m\leq n+1$.\ The maximality of the family $A_1,\ldots, A_m$ implies that $A_j$ is a $p$-atom, for $j=1,\ldots, m$, and 
\[
\sum_{j=1}^mp(A_j)=p\bigg(\bigcup_{j=1}^mA_j\bigg)=1,
\]
so that $p=\sum_{j=1}^m\alpha_j\, p_{A_j}$ with $\alpha_j:=p(A_j)>0$ for $j=1,\ldots, m$.\ Moreover, since $f_i$ is $p_{A_j}$-a.s.\ equal to $f_i(A_j)$ for $i=1,\ldots, n$, we have
\[
u_j=\bigg(p(A_j),\int_{A_j}f_1\,\d p,\ldots,\int_{A_j}f_n\,\d p\bigg)=\alpha_j v_j\quad \text{for } j=1,\ldots, m.
\]
Since $p$ is an extreme point, the vectors $u_1,\ldots, u_m$ are linearly independent, and, since $\alpha_1,\ldots, \alpha_m>0$, so are the vectors $v_1,\ldots, v_m$.

\noindent\ref{it:moment:atomic}$\,\Rightarrow\,$\ref{it:moment:minimal}:\ Let $p'\in\cM$ with $p'\ll p$.\ Since $A_1,\ldots, A_m$ are pairwise disjoint and  $p=\sum_{j=1}^m\alpha_j\, p_{A_j}$, we have $p(A_j)=\alpha_j$ for $j=1,\ldots, m$ and $\sum_{j=1}^m\alpha_j=1$, so that $p$ is concentrated on $\bigcup_{j=1}^mA_j$, and $p'\ll p$ implies that so is $p'$.\ Since $A_j$ is a $p$-atom, the density $\frac{\d p'}{\d p}$ is $p$-a.s.\ equal to a constant $c_j\in[0,\infty)$ on $A_j$, so that
\[
p'(B\cap A_j)=\int_{B\cap A_j}\frac{\d p'}{\d p}\,\d p=c_j\, p(B\cap A_j)=\beta_j\, p_{A_j}(B),\quad\text{for } B\in\cF,
\]
with $\beta_j:=c_j\alpha_j$ for $j=1,\ldots, m$, and therefore $p'=\sum_{j=1}^m\beta_j\, p_{A_j}$.\ The constraints $\int_\Omega f_i\,\d p'=a_i=\int_\Omega f_i\,\d p$, for $i=1,\ldots, n$, then give
\[
\sum_{j=1}^m(\beta_j-\alpha_j)v_j=0,
\]
and linear independence yields $\beta_j=\alpha_j$ for all $j=1,\ldots, m$, so that $p'=p$.\ Hence, $p$ is $\ll$-minimal in $\cM$.

\noindent\ref{it:moment:minimal}$\,\Rightarrow\,$\ref{it:moment:extreme}:\ Every $\ll$-minimal element of $\cM$ is an extreme point of $\cM$ by Theorem~\ref{thm:extreme:divergence}.

\noindent The remaining equivalence \ref{it:moment:minimal}$\,\Leftrightarrow\,$\ref{it:moment:blowup} follows from Theorem~\ref{thm:minimal:divergence}, which completes the proof.
\end{proof}

In the case where the extreme points of $\cP$ are the Dirac measures,\footnote{We denote by $\delta_\omega$ the Dirac measure concentrated at $\omega\in \Omega$, i.e., for $A \in \cF$, $\delta_\omega(A) = 1$ if $\omega \in A$, and $\delta_\omega(A) = 0$ otherwise.} Theorem~\ref{thm:moment:extreme} yields the famous characterization of extreme points of sets of probability measures with generalized moment constraints by \cite{winkler1988extreme}, see also \cite{karr1983extreme}.\
We point out that the proofs in \cite{karr1983extreme} and \cite{winkler1988extreme} rely on a more elaborated functional-analytic result by \cite{douglas1964extremal, douglas1966extremal-II} about the relation between extremality of a measure and the density of a certain space of functions in the $L^p$-space of this measure.

\begin{corollary}[Winkler (1988)]\label{cor:winkler}
Assume that $\cI=\{1,\ldots,n\}$ and that the extreme points of $\cP$ are exactly the Dirac measures.\ Then, $p\in \cM$ is an extreme point of $\cM$ if and only if
\[
p=\sum_{j=1}^m\alpha_j\delta_{\omega_j}
\]
with $m\in \N$ satisfying $m\leq n+1$, $\alpha_1,\ldots, \alpha_m>0$, and pairwise distinct $\omega_1,\ldots, \omega_m\in \Omega$, such that the
vectors
\[
v_j:=\big(1,f_1(\omega_j),\ldots,f_n(\omega_j)\big)^\top\in\R^{n+1},
\qquad j=1,\ldots,m,
\]
are linearly independent.
\end{corollary}

\begin{proof}
Let $p$ be an extreme point of $\cM$.\ The representation $p=\sum_{j=1}^m\alpha_j\delta_{\omega_j}$ with linearly independent vectors $v_1,\ldots, v_m$ follows from Theorem~\ref{thm:moment:extreme} since $p_{A_j}$ is an extreme point of $\cP$, so that there exist $\omega_j\in \Omega$ with $p_{A_j}=\delta_{\omega_j}$ for $j=1,\ldots, m$, which implies that $f_i(A_j)=f_i(\omega_j)$ for $i=1, \dots, n$ and $j = 1, \dots, m$.

Conversely, let $p=\sum_{j=1}^m\alpha_j\delta_{\omega_j}$ with pairwise distinct $\omega_1,\ldots, \omega_m\in \Omega$ and linearly independent vectors $v_1,\ldots, v_m$.\ For all $j,k=1,\ldots, m$ with $j\neq k$, linear independence of $v_1,\ldots, v_m$ implies $v_j\neq v_k$, and therefore  $f_i(\omega_j)\neq f_i(\omega_k)$ for some $i\in\{1,\ldots, n\}$.\ Consequently, the sets 
\[
A_j:=\bigcap_{i=1}^n f_i^{-1}\big(\{f_i(\omega_j)\}\big)\in \cF, \qquad j=1,\ldots, m,
\]
are pairwise disjoint with $\delta_{\omega_j}(A_j)=1$, which implies that $A_j$ is a $p$-atom with $p_{A_j}=\delta_{\omega_j}$ and $f_i(A_j)=f_i(\omega_j)$ for all $i=1,\ldots, n$ and $j=1,\ldots, m$. Hence, condition \ref{it:moment:atomic} in Theorem~\ref{thm:moment:extreme} is satisfied, and $p$ is an extreme point of $\cM$.

\end{proof}

\subsection{Multi-marginal couplings} \label{sec:couplings}
Let $(\Omega_i, \cF_i, p_i)_{i = 1, \dots, n}$ be a family of probability spaces. Let $\Omega = \Omega_1 \times \dots \times \Omega_n$ be endowed with the product $\sigma$-algebra $\cF$, and consider
\[
\Pi(p_1, \dots, p_n) := \big\{ p \in \cP \,\big|\, p \circ \operatorname{pr}_i^{-1} = p_i \text{ for all } i = 1, \dots, n\big\},
\]
where $\operatorname{pr}_i\colon \Omega\to \Omega_i$ denotes the $i$-th coordinate projection for $i=1,\ldots, n$. The set $\Pi(p_1, \dots, p_n)$ therefore describes the set of all probability measures $p$ on $(\Omega, \cF)$ whose marginals on $\Omega_i$ coincide with $p_i$, also known as the set of \emph{multi-marginal couplings} of $p_1, \dots, p_n$, see, e.g., \cite{zbMATH06527378} for an overview on multi-marginal optimal transport.\ Note that this setup presents a special instance of integral constraints by considering, for each $i \in \{1, \dots, n\}$ and $A_i \in \cF_i$, the indicator function $f_{(i,A_i)} := \one_{\operatorname{pr}_i^{-1}(A_i)}$ and value $a_{(i,A_i)} := p_i(A_i)$.

If $\cF_i$ is generated by a countable $\pi$-system $\mathcal{G}_i=\{A_i^k\mid k\in\N\}$ for $i=1,\ldots, n$, we only have countably many integral constraints
\[
\int_\Omega \one_{\operatorname{pr}_i^{-1}(A_i^k)}\,\d p'
=p_i(A_i^k),
\qquad i=1,\ldots,n,\quad k\in\N.
\]
Indeed, equality of the two marginals on $\mathcal{G}_i$ extends to equality on $\cF_i$ by the $\pi$--$\lambda$ theorem.\ Moreover, if the generators are finite, the set of multi-marginal couplings is therefore a finite moment set to which Theorem~\ref{thm:moment:extreme}
applies.\ 

In the finite-state setting, Theorem~\ref{thm:finite:divergence} allows us to obtain the following alternative characterization, cf.\ \citet[Theorem 4]{bauch2025correlationuncertaintydecisiontheoreticapproach} for the case $q=p_1\otimes\cdots\otimes p_n$.

\begin{corollary}
    Let $\Omega_1,\ldots, \Omega_n$ be finite sets, each carrying its power set as a $\sigma$-algebra.\ Moreover, let $\cM:=\Pi(p_1,\ldots,p_n)$ be the set of multi-marginal couplings of $p_1,\ldots,p_n$, let $q\in \cM$ with $p'\ll q$ for all $p'\in \cM$, and let $\varphi\colon[0,\infty)\to\R$ be a convex function satisfying $\varphi(1)=0$ and $\varphi_+'(0)=-\infty$.\ Then, for $p\in\cM$, the following statements are equivalent.
    \begin{enumerate}
        \item $p$ is an extreme point of $\cM$.
        \item $p$ is a strict local maximizer of the function
        \begin{equation}
            \cM \to \R,\quad p' \mapsto \Div{\varphi}{p'}{q}.
        \end{equation}
    \end{enumerate}
\end{corollary}

In order to partially recover Theorem \ref{thm:finite:divergence} for multi-marginal couplings with arbitrary measurable spaces $(\Omega_1, \cF_1),\ldots, (\Omega_n,\cF_n)$, we need to impose additional conditions.\ For this, we introduce the notion of marginal uniqueness sets, which has been used in the framework of bivariate couplings since \cite{letac1966representation}.
\begin{definition}
    We call a set $S\in\cF$ a \emph{marginal uniqueness set} for $\Pi(p_1,\ldots,p_n)$ if there is at most one element of $\Pi(p_1,\ldots,p_n)$ concentrated on $S$.
\end{definition}
In \cite{benes1987support,hestir1995supports,moameni2016supports}, the authors study the aperiodic structure of the supports of extreme points of bivariate couplings, and give conditions under which such a support is a marginal uniqueness set. We also refer to the collection \cite{dallaglio1991advances} for a dated but thorough discussion of this topic.

\begin{proposition}\label{prop:uniqueness:minimal}
    Let $(\Omega_i,\cF_i,p_i)_{i=1,\ldots, n}$ be a family of probability spaces,  $\cM:=\Pi(p_1,\ldots,p_n)$ be the set of multi-marginal couplings of $p_1,\ldots,p_n$, and $p\in\cM$ be concentrated on a marginal uniqueness set $S\in\cF$.\ Then $p$ is $\ll$-minimal in $\cM$.\ In particular, $p$ is an extreme point of $\cM$, and the blow-up \eqref{eq:blowup} of Theorem~\ref{thm:minimal:divergence} holds for every $p'\in\cM\setminus\{p\}$, every $q\in\cM$ with $p,p'\ll q$, and every convex function $\varphi$ satisfying \eqref{eq:phi:blowup}.
\end{proposition}

\begin{proof}
    Let $p'\in\cM$ with $p'\ll p$. Then $p(\Omega\setminus S)=0$ forces $p'(\Omega\setminus S)=0$, so $p'$ is concentrated on the marginal uniqueness set $S$. Since $p'\in\Pi(p_1,\ldots,p_n)$, the uniqueness property gives $p'=p$. Hence $p$ is $\ll$-minimal in $\cM$, and the remaining assertions follow from Theorem~\ref{thm:minimal:divergence}.
\end{proof}

We now collect some classes of couplings concentrated on marginal uniqueness sets, to which Proposition~\ref{prop:uniqueness:minimal} applies.

\begin{example}\label{ex:uniqueness}
\begin{enumerate}[wide, labelindent=0pt, label=\textup{(\roman*)}]
    \item \emph{Monge transport couplings.}\ A simple sufficient condition for the hypothesis of Proposition~\ref{prop:uniqueness:minimal} is that $p$ is concentrated on the graph of a measurable map. Suppose $T_2\colon\Omega_1\to\Omega_2,\ldots,T_n\colon\Omega_1\to\Omega_n$ are measurable and
    \[
        S:=\big\{(\omega_1,T_2(\omega_1),\ldots,T_n(\omega_1))\,\big\vert\,\omega_1\in\Omega_1\big\}\in\cF.
    \]
    Then, $S$ is a marginal uniqueness set for $\Pi(p_1,\ldots,p_n)$. As a matter of fact, any $p\in\Pi(p_1,\ldots,p_n)$ concentrated on $S$ satisfies, for all $A_i\in\cF_i$,
    \begin{align*}
        p(A_1\times\cdots\times A_n)&=p\big(\big(A_1\cap T_2^{-1}(A_2)\cap\cdots\cap T_n^{-1}(A_n)\big)\times\Omega_2\times\cdots\times\Omega_n\big)\\
        & =p_1\big(A_1\cap T_2^{-1}(A_2)\cap\cdots\cap T_n^{-1}(A_n)\big),
    \end{align*}
    so $p$ is determined by its first marginal $p_1$.\ In particular, every deterministic (Monge-type) coupling $p=(\mathrm{id},T_2,\ldots,T_n)_{\#}p_1$ with measurable graph satisfies the hypothesis of Proposition~\ref{prop:uniqueness:minimal}. For $n=2$, more general sufficient conditions on the support are available in the literature, see \cite{hestir1995supports} and \cite{moameni2016supports}, where unions of countably many graphs with a suitable acyclic structure are shown to be marginal uniqueness sets.\label{it:monge}
    
    \item \emph{Optimal transport couplings.} Let $\Omega_1=\Omega_2=\R^d$, and let $p_1,p_2$ be probability measures on $\R^d$ with finite second moments and such that $p_1$ is absolutely continuous with respect to the Lebesgue measure. By Brenier's theorem~\citep{brenier1991polar}, the optimal coupling of $p_1$ and $p_2$ for the quadratic cost is $p=(\mathrm{id},\nabla\eta)_{\#}p_1$ for a convex function $\eta\colon\Rd\to\R\cup\{+\infty\}$, which is differentiable on a set $D \subseteq \Rd$ with $p_1(D)=1$. Convexity of $\eta$ implies that $D$ is a Borel measurable set and that $\nabla \eta$ is Borel measurable. Therefore, the graph
    \[
        S:=\big\{\big(x,\nabla\eta(x)\big)\,\big\vert\, x\in D\big\}
    \]
    is a Borel subset of $\R^d\times\R^d$. Hence $p$ is a Monge transport coupling as in \ref{it:monge}, so $S$ is a marginal uniqueness set for $\Pi(p_1,p_2)$.\ Proposition~\ref{prop:uniqueness:minimal} then applies: $p$ is $\ll$-minimal in $\Pi(p_1,p_2)$, hence an extreme point, and Theorem~\ref{thm:minimal:divergence} yields the blow-up of every $\varphi$-divergence satisfying \eqref{eq:phi:blowup} with respect to every dominating measure $q$. More generally, for any cost of the form $c(x,y)=h(x-y)$ with $h$ strictly convex, \citet[Theorem 1.2]{gangbo1996geometry} yields a unique optimal plan concentrated on the graph of a map whenever $p_1$ is absolutely continuous with respect to the Lebesgue measure, which is again covered by~\ref{it:monge}.\label{it:brenier}

    \item \emph{Monotone sets in dimension one.} Let $\Omega_1 = \Omega_2 = \R$. Then every monotone Borel set $S\subseteq\R^2$, i.e., a set such that $(x-x')(y-y')\geq 0$ for all $(x,y),(x',y')\in S$, is a marginal uniqueness set, without any graph structure. As a matter of fact, a coupling concentrated on a monotone set coincides with the comonotone coupling, cf.\ \citet[Proposition 4.5]{ambrosio2021lectures}, which is uniquely determined by the marginals. In particular, Proposition~\ref{prop:uniqueness:minimal} applies to every coupling concentrated on a monotone set, regardless of the absolute continuity of the marginals.\label{it:monotone}
\end{enumerate}
\end{example}

The marginal-uniqueness-set hypothesis in Proposition~\ref{prop:uniqueness:minimal} is sufficient but not necessary for extremality.\ In fact, \citet{losert1982counterexamples} provides an example on the unit square of an extreme doubly stochastic measure $p$ and a distinct doubly stochastic measure $p'$ with $p'\ll p$. In the terminology of Definition~\ref{def:minimal}, this is precisely an extreme coupling that is not $\ll$-minimal. Such a measure cannot be concentrated on any marginal uniqueness set, because absolute continuity would force $p'$ to be concentrated on the same set and the uniqueness property would then imply $p'=p$.\ We therefore see that Proposition~\ref{prop:uniqueness:minimal} does not characterize all extreme couplings.

\subsection{Dominated sets of measures} For a fixed probability measure $q\in \cP$, we consider the set of all probability measures that are dominated by $q$, i.e.,
\begin{align*}
    \cP_q = \{ p \in \cP \, \mid \, p\ll q \}.
\end{align*}
Denoting by $\cL_q$ the set of all signed measures with finite variation that are absolutely continuous with respect to $q$, we have $\cP_q=\cP\cap\cL_q$.\ Note that $\cL_q$ is a linear subspace of the space of signed measures with finite variation, therefore the set $\cP_q$ is of the form covered by this paper.

The extreme points of this set are determined exactly by the atoms of the dominating measure.\ For a $q$-atom $A\in \cF$, we denote by $q_A$ the measure defined as in \eqref{eq:atom:measure}.

\begin{proposition}\label{prop:dominated}
The extreme points of $\cP_q$ are precisely the measures $q_A$, where $A$ ranges over all $q$-atoms.\ In particular, $\cP_q$ has no extreme points if $q$ is atomless.
\end{proposition}

\begin{proof}
If $A\in \cF$ is a $q$-atom, every probability measure $p'\ll q_A$ equals $q_A$. Hence, condition~\ref{it:thm:bounded} in Theorem~\ref{thm:extreme:divergence} shows that $q_A$ is extreme in $\cP_q$.

Conversely, let $p\in\cP_q$ be extreme.\ If there was a set $B\in\cF$ with $0<p(B)<1$, then the conditional probability $p':=\frac{p(\,\cdot\;\cap B)}{p(B)}$ would belong to $\cP_q\setminus\{p\}$ with bounded density
\[
\frac{\d p'}{\d p}=\frac{\one_B}{p(B)},
\]
contradicting item~\ref{it:thm:bounded} in Theorem~\ref{thm:extreme:divergence}. Thus, $p$ is $\{0,1\}$-valued.

Now, let $f:=\d p/\d q$ and $A:=\{f>0\}$, so that $p(A)=1$.\ For every measurable $B\subseteq A$, either $p(B)=0$, in which case $q(B)=0$, or $p(B)=1$, in which case $q(A\setminus B)=0$.\ Therefore, the set $A$ is a $q$-atom.\ Since every measurable function is $q$-a.s.\ constant on a $q$-atom and $\int_A f\,\d q=1$, one has $f=\frac{\one_A}{q(A)}$ $q$-a.s., and therefore $p=q_A$.
\end{proof}

More generally, we can again consider the intersection of the set $\cP_q$ with an affine set $H$ of signed measures with finite variation, which is again assumed to be nonempty.\ Observe that $H\cap\cL_q$ is still affine and
\[
\cP_q\cap H=\cP\cap\big(H\cap\cL_q\big)=(\cP\cap H)\cap \cP_q.
\]
 The constraint of absolute continuity does not create new extreme points as the following proposition shows.

\begin{proposition}\label{prop:dominated:intersection}
Let $q\in\cP$ and $H$ be an affine set of signed measures with finite variation.\ Then
\[
\ex\big(\cP_q\cap H\big)=\ex\big(\cP\cap H\big)\cap\cP_q.
\]
Moreover, $p\in\cP_q\cap H$ is $\ll$-minimal in $\cP_q\cap H$ if and only if it is $\ll$-minimal in $\cP\cap H$.
\end{proposition}

\begin{proof}
Let $p\in\cP_q\cap H$.\ Then, for $p'\in\cP\cap H$ with $p'\ll p$ and $p$-a.s.\ bounded density $\frac{\d p'}{\d p}$, it follows that $p'\ll p\ll q$, so that $p'\in\cP_q\cap H$.\ Conversely, every $p'\in\cP_q\cap H$ belongs to $\cP\cap H$.\ Hence, condition~\ref{it:thm:bounded} in Theorem~\ref{thm:extreme:divergence} holds for $p'\in \cP_q\cap H$ if and only if it holds for $p'\in \cP\cap H$, which gives the first assertion.\ The second assertion follows from the same observations.
\end{proof}

\begin{remark}\label{rem.atomless}
Recall that if $q\in \cP$ is atomless and $p\in \cP_q$, it follows that $p$ is atomless.\ Indeed, let $q\in \cP$ be atomless, $p\in \cP_q$, and assume that $A\in \cF$ is a $p$-atom.\ Then the density $\frac{\d p}{\d q}$ is $p$-a.s.\ constant on $A$, i.e., there exists $c> 0$ such that
\[
 p\bigg(A\cap \bigg\{\frac{\d p}{\d q}=c\bigg\}\bigg)=p(A).
\]
Otherwise, there would exist $ c\geq 0$ with
\[
 p\bigg(A\cap \bigg\{\frac{\d p}{\d q}\leq  c\bigg\}\bigg)>0\quad \text{and}\quad p\bigg(A\cap \bigg\{\frac{\d p}{\d q}> c\bigg\}\bigg)>0,
\]
which would contradict the fact that $A$ is a $p$-atom.\
Since $p\ll q$, the set $A_0:=A\cap \big\{\frac{\d p}{\d q}=c\big\}$ therefore has positive $q$-measure, and since $q$ is atomless, there exists $B\in \cF$ with $B\subset A_0$ and $0<q(B)<q(A_0)$, which implies
$$
0<cq(B)=p(B)\quad \text{and}\quad p(B)=cq(B)<cq(A_0)=p(A_0)= p(A),
$$
contradicting the fact that $A$ is a $p$-atom. 

This observation allows to deduce the following result from Theorem~\ref{thm:moment:extreme} and Proposition~\ref{prop:dominated:intersection}.
\end{remark}

\begin{corollary}\label{cor:atomless}
Let $q\in\cP$ be atomless and $H$ be the affine set determined by $n\in\N$ integral constraints of the form $\int_\Omega f_i\,\d p=a_i$ with measurable functions $f_i \colon \Omega \to \R$ and $a_i \in \R$ for $i=1,\ldots, n$.\ Then $\cP_q\cap H$ has no extreme points.
\end{corollary}

\begin{proof}
Since $q$ is atomless, every $p\in\cP_q$ is atomless, cf.\ Remark \ref{rem.atomless}.\ By Theorem~\ref{thm:moment:extreme}, every extreme point of $\cP\cap H$ is a finite convex combination of measures $p_{A_j}$ associated with $p$-atoms $A_j$, and therefore admits atoms.\ Proposition~\ref{prop:dominated:intersection} then yields that $\cP_q\cap H$ has no extreme points.
\end{proof}

We conclude this section with an example from mathematical finance motivating an extension to the framework of sets of equivalent measures.\
For a one-period market $(\Omega,\cF,q)$ with a $q$-integrable payoff vector $X=(X_1,\ldots,X_d)$, by the fundamental theorem of asset pricing, leaving aside technicalities, no-arbitrage is equivalent to the nonemptiness of the set of equivalent martingale measures
\[
\cM
:=\bigg\{p\in\cP\,\bigg|\,p\approx q\text{ and } X_i\in L^1(p)\text{ with }\int_\Omega X_i\,\d p=0 \text{ for }i=1,\ldots, d\bigg\}.
\]
This set is, in general, not an intersection of $\cP$ with an affine set, so that the results of Section~\ref{sec:extreme} do not apply to it directly.\ Nevertheless, its extreme points can be traced back to those of $\cP_q\cap H$, as the following proposition shows.

\begin{proposition}\label{prop:equivalent}
Let $q\in\cP$, $H$ be an affine set of signed measures with finite variation, and
\[
\cM:=\{p\in\cP_q\cap H\mid p\approx q\}.
\]
Then, every extreme point of $\cM$ is an extreme point of $\cP_q\cap H$.
\end{proposition}

\begin{proof}
Let $p\in\cM$ and suppose that $p$ is not an extreme point of $\cP_q\cap H$.\ By Theorem~\ref{thm:extreme:divergence}, there exists $p'\in(\cP_q\cap H)\setminus\{p\}$ with $\frac{\d p'}{\d p}\leq C$ $p$-a.s.\ for some constant $C\geq 1$.\ For $\varepsilon\in(0,1)$ with $\varepsilon(C-1)<1$, consider
\[
p_\pm:=(1\pm\varepsilon)p\mp\varepsilon p'.
\]
Since the coefficients sum to one, $p_\pm\in H$, and $p=\frac12(p_++p_-)$.\ Writing $f:=\frac{\d p}{\d q}$ and $f':=\frac{\d p'}{\d q}$, the measure $p_-$ has $q$-density $(1-\varepsilon)f+\varepsilon f'$, which is strictly positive $q$-a.s.\ because $p\approx q$ implies $f>0$ $q$-a.s. Similarly, $p_+$ has $q$-density
\[
(1+\varepsilon)f-\varepsilon f'=f\bigg((1+\varepsilon)-\varepsilon\frac{\d p'}{\d p}\bigg)>0\quad q\text{-a.s.}
\]
by the choice of $\varepsilon$.\ Hence, $p_-$ and $p_+$ are distinct probability measures equivalent to $q$, yielding $p_\pm\in\cM$, and therefore $p$ is not an extreme point of $\cM$.
\end{proof}

Coming back to the one-period market and the set of equivalent martingale measures $\cM$, we note that it is of the form
\[
\cM = \{p \in \cP_q \cap H \mid p \approx q\},
\]
where $H$ denotes the affine set determined by the constraints $\int_\Omega X_i\,\d p=0$ for $i=1,\ldots, d$.\ Now, if $q$ is atomless, Corollary~\ref{cor:atomless} shows that $\cP_q\cap H$ has no extreme points, and, therefore, Proposition~\ref{prop:equivalent} yields that the set of equivalent martingale measures $\cM$ has no extreme points either.

\subsubsection*{Acknowledgements}
The first-named author gratefully acknowledges financial support by the German Research Foundation (DFG) via RTG2865/1 -- 492988838.\ The second-named author was funded by the Natural Sciences and Engineering Research Council of Canada via Discovery Grant no.\ RGPIN-2025-04219.

\bibliographystyle{chicago}

\end{document}